\documentclass[12pt]{article}

\usepackage{amsmath,amsthm,amssymb,amsfonts}
\usepackage{stmaryrd}
\usepackage{graphicx}
\usepackage{enumerate}
\usepackage{bbm}

\newtheorem{theorem}{Theorem}
\newtheorem{lemma}[theorem]{Lemma}
\newtheorem{proposition}[theorem]{Proposition}
\newtheorem{corollary}[theorem]{Corollary}
\newtheorem{algorithm}[theorem]{Algorithm}

\title{Directional descent}
\author{Andrew J. Young}
\date{}

\begin{document}

\maketitle

\begin{abstract}
We identity the optimal non-infinitesimal direction of descent for a convex function.
An algorithm is developed that can theoretically minimize a subset of (non-convex) functions.
\end{abstract}

\section{Introduction}

Non-smooth convex functions can be minimized using various algorithms \cite{Bertsekas2015} (more efficiently with additional structure \cite{Nesterov2005,Bertsekas2015}).
Non-convex optimization problems are in general NP-hard but, with some restrictions, e.g. strong smoothness properties and requiring only $\varepsilon$-stationary solutions, there are efficient algorithms \cite{Nesterov2006,Carmon2017,Clement2020} and complexity bounds \cite{Cartis2012,Carmon2020}.
We consider both non-convexity and limited smoothness.

Our algorithm is two-staged: the first stage is finding the minimizer of a constrained optimization problem and the second is an incrementation step.
More specifically, the first stage seeks the minimizer of a function over a closed ball with some positive radius.
The solution of this optimization is the optimal direction of descent for any ball with positive radius (Theorem \ref{theorem_optimal_direction}).
The second stage proceeds from the initial point along the given direction.
If the function is Lipschitz, then given an appropriately close solution in the first stage, e.g. order $\varepsilon$, the second stage provides an $\varepsilon$-optimal solution with iteration complexity $O(1/\varepsilon)$.

This algorithm can be applied to any function that is lower semicontinuous with a bounded effective domain and unique minimizer whose points of convexity in its effective domain have a nonempty interior.

We begin with some auxiliary results related to lower convex envelopes.

\section{Lower convex envelope}

Let $X$ be a normed vector space and $f : X \rightarrow (-\infty,\infty]$.
The lower convex envelope of $f$ is the pointwise supremum of all convex functions lower bounding $f$, i.e.
\begin{equation*}
\sup \{ h(x) : h \le f \text{ convex} \},
\end{equation*}
and, as the pointwise supremum of an arbitrary collection of convex functions is convex, it is convex.
In the sequel, the lower convex envelope of a function $f$ is denoted by $\breve{f}$.

We give an explicit characterization of the lower convex envelope.
This representation is well known and can be found in \cite{Rockafellar1970} on page 36.

\begin{lemma}
\label{lemma_lower_convex_envelope}
The lower convex envelope of $f$ is
\begin{equation*}
g(x) = \inf \left \{ \sum_{k=1}^{m} \lambda_{k} f(x_{k}) : (\underline{\lambda},\underline{x}) \in C(x,f) \right\},
\end{equation*}
where 
\begin{equation*}
C(x,f)
:= \left\{ (\underline{\lambda},\underline{x}) : \sum_{k=1}^{m} \lambda_{k} x_{k} = x, m \in \mathbb{N}, \lambda_{k} \ge 0, \sum_{k=1}^{m} \lambda_{k} \! = \! 1, x_{k} \in \mathrm{dom}(f) \right\},
\end{equation*}
$\mathrm{dom}(f)$ is the effective domain of $f$ and the infimum of the empty set is $\infty$.
\end{lemma}

$f$ is \emph{convex at a point} $x$, if, for all $(\underline{\lambda},\underline{x}) \in C(x,f)$,
\begin{equation*}
f(x) \le \sum \lambda_{k} f(x_{k}).
\end{equation*}
This definition includes points $x$ outside $\mathrm{conv}(\mathrm{dom}(f))$ which have $f$ infinite, empty $C(x,f)$ and are thus, trivially, convex.
Moreover, If $x$ is an extreme point of $\mathrm{conv}(\mathrm{dom}(f))$ then $C(x,f) = \{(1,x)\}$ and thus $x$ is a point of convexity for $f$.
Let $A_{f}$ be the \emph{set of points where $f$ is convex}.
We have the following immediate corollary to Lemma \ref{lemma_lower_convex_envelope}.

\begin{corollary}
\label{corollary_lce_points}
\begin{equation*}
A_{f} = \{ x : f(x) = \breve{f}(x) \}.
\end{equation*}
\end{corollary}

\section{Euclidean space}

We now specialize to the case where $X=\mathbb{R}^{n}$ is $n$-dimensional Euclidean space endowed with the Euclidean norm and $f : \mathbb{R}^{n} \rightarrow (-\infty,\infty]$.

\subsection{Sufficiency of the lower convex envelope}

Suppose $f$ is not convex but has a \emph{unique} minimizer.
Any minimizer of $f$ is a minimizer of the lower convex envelope $\breve{f}$, but the converse is not necessarily true.
We give sufficient conditions for $\breve{f}$ to have the same unique minimizer.
Necessity of these conditions is also discussed.

\begin{theorem}
If $f$ is lower semicontinuous with a bounded effective domain and unique minimizer, then the lower convex envelope of $f$ has the same unique minimizer.
\end{theorem}

\begin{proof}
Let $x^{\ast}$ be the unique minimizer of $f$, so that $f(x) > f(x^{\ast})$ for all $x \ne x^{\ast}$.
By Caratheodory (Theorem \ref{theorem_caratheodory}) the sets $C(x,f)$ can be restricted to vectors of length $n+1$.
Choose by boundedness any compact set $B$ containing $\mathrm{dom}(f)$.
Let $g(\underline{\lambda},\underline{x}) = \sum_{k=1}^{n+1} \lambda_{k} f(x_{k})$ on $[0,1]^{n+1} \times B^{n+1}$ with the convention that $0 \cdot \infty = 0$.
Then (by properties of $\liminf$) $g$ is lower semicontinuous (lsc).
Let
\begin{equation*}
D_{x}
= \left\{ (\underline{\lambda},\underline{x}) : \sum_{k=1}^{n+1} \lambda_{k} x_{k} = x, \lambda_{k} \ge 0, \sum_{k=1}^{n+1} \lambda_{k} = 1, x_{k} \in B \right\}.
\end{equation*}
Then $D_{x}$ is compact, as $[0,1]^{n+1} \times B^{n+1}$ is compact if $B$ compact, the functions $(\underline{\lambda},\underline{x}) \mapsto \sum_{k=1}^{n+1} \lambda_{k} x_{k}$ and $(\underline{\lambda},\underline{x}) \mapsto \sum_{k=1}^{n+1} \lambda_{k}$ are continuous and the intersection of closed sets, the preimage of $x$ and $1$, respectively, with a compact set is compact.
By construction, the infimum in Lemma \ref{lemma_lower_convex_envelope} can be extended to $D_{x}$ instead of $C(x,f)$.
As in that lemma, if $D_{x}$ is empty $\breve{f}(x)$ is infinite.
The infimum of an lsc function on a nonempty compact set $g : D_{x} \rightarrow (-\infty,\infty]$ is achieve by some $(\underline{\lambda},\underline{x})$.
Then, for any such $(\underline{\lambda},\underline{x})$,
\begin{equation*}
\sum_{k=1}^{n+1} \lambda_{k} f(x_{k}) = f(x^{\ast})
\quad \Longleftrightarrow \quad
\lambda_{k} > 0 \Longrightarrow f(x_{k}) = f(x^{\ast}) \Longrightarrow x_{k} = x^{\ast}.
\end{equation*}
Hence $(\underline{\lambda},\underline{x}) \in D_{x^{\ast}}$.
\end{proof}

Necessity of lower semicontinuity:
\begin{equation*}
f(x) 
=
\begin{cases}
0 & x = 0 \\
x - 1/2 & 1/2 < x \le 1 \\
\infty & \text{else}
\end{cases}.
\end{equation*}

Necessity of bounded effective domain:
\begin{equation*}
f(x) 
=
\begin{cases}
\lvert x \rvert & -1 \le x \le 1 \\
1 & \text{else}
\end{cases}.
\end{equation*}

\subsection{Existence of subgradients}

It is well known that a proper convex function has subgradients at all the interior points of its effective domain (Proposition \ref{proposition_subgradient_existence}).
We explore the connection between points of convexity and existence of subgradients.

\begin{lemma}
\hspace{1mm}
\begin{enumerate}[i)]
\item
If $f$ has a subgradient at $z$, then $f$ is convex at $z$.
\item
If $\breve{f}$ is proper, $z \in \mathrm{int}(\mathrm{dom}(f))$ and $f$ is convex at $z$, then $f$ has a subgradient at $z$.
\end{enumerate}
\end{lemma}

\begin{proof}
$(i)$
If $f$ has a subgradient at $z$, then there exists $g$ such that
\begin{equation*}
f(x) \ge g'(x-z) + f(z),
\end{equation*}
for all $x$.
In particular, if
\begin{equation*}
z = \sum \lambda_{k} z_{k},
\end{equation*}
then, for all $k$,
\begin{equation*}
f(z_{k}) \ge g'(z_{k} - z) + f(z).
\end{equation*}
Summing over $k$
\begin{equation*}
\sum \lambda_{k} f(z_{k})
\ge \sum \lambda_{k} \left[ g'(z_{k} - z) + f(z) \right]
= g'\left( \sum \lambda_{k} z_{k} - z \right) + f(z)
= f(z).
\end{equation*}
$(ii)$
Let $h$ be the lower convex envelope of $f$.
Then by Proposition \ref{proposition_subgradient_existence}, there exists a subgradient $h_{z}$ for all $z \in \mathrm{int}(\mathrm{dom}(h))$, where $\mathrm{dom}(h) = \mathrm{conv}(\mathrm{dom}(f))$, such that
\begin{equation*}
h(x) \ge h_{z}'(x-z) + h(z)
\end{equation*}
for all $x$.
By definition of $h$ and Corollary \ref{corollary_lce_points}, for all such $z$ where $f$ is convex,
\begin{equation*}
f(x) \ge h(x)
\ge h_{z}'(x -z) + h(z)
= h_{z}'(x -z) + f(z) 
\end{equation*}
for all $x$.
\end{proof}

\subsection{Optimal direction}

If $f$ have a unique minimizer $x^{\ast}$, then, for all $x_{0} \ne x^{\ast}$ and $\delta > 0$, there exists $\lVert d \rVert \le \delta$ such that
\begin{equation*}
\inf_{\alpha > 0} f( x_{0} + \alpha d ) = f(x^{\ast}),
\end{equation*}
namely $d = \delta d_{0}$ with
\begin{equation*}
d_{0} := \frac{x^{\ast} - x_{0}}{\lVert x^{\ast} - x_{0} \lVert }.
\end{equation*}
This choice is in fact optimal, a result that forms the theoretical basis for our algorithm.

\begin{theorem}
\label{theorem_optimal_direction}
Let $f$ be convex with a unique minimizer $x^{\ast}$.
For all $x_{0} \ne x^{\ast}$ in $\mathrm{dom}(f)$ and directions $\lVert d \rVert \le \delta \le \lVert x^{\ast} - x_{0} \rVert$
\begin{equation*}
f(x_{0} + d) \ge f(x_{0} + \delta d_{0}).
\end{equation*}
If $d \ne \delta d_{0}$, then this inequality is strict.
\end{theorem}

\begin{proof}
Obvious for $\delta = 0$.
Suppose $\delta > 0$.
WLOG $x^{\ast} = 0$, otherwise, as this property is shift invariant, consider
\begin{equation*}
g(x) = f(x+x^{\ast}),
\end{equation*}
and $\lVert d \rVert = \delta$, by convexity the minimum over such $d$ must occur on the boundary.

Let
\begin{equation*}
f_{0}(x)
=
\begin{cases}
f\left( \lVert x \rVert \frac{ x_{0} }{\lVert x_{0} \rVert } \right) & \lVert x \rVert \le \lVert x_{0} \rVert \\
\infty & \text{else}
\end{cases},
\end{equation*}
which is convex by Lemma \ref{lemma_strictly_decreasing}.
Given two functions $h_{1}$ and $h_{2}$ let $h_{1} \sim h_{2}$ if $h_{1},h_{2}$ are convex, $h_{1},h_{2} \ge f(0)$ and $h_{1}(x) = h_{2}(x)$ for all $x = \alpha x_{0}$ and $0 \le \alpha \le 1$, e.g. $f_{0} \sim f$.
The equivalence class $[f]$ is partially ordered under the pointwise order, i.e. $h_{1} \le h_{2}$ if $h_{1}(x) \le h_{2}(x)$ for all $x$.
Moreover, if the pointwise limit of proper convex functions is proper, then it is convex.
Therefore, any chain in $[f]$ has a lower bound, i.e. the infimum of a nonincreasing net of extended real numbers bounded below converges to a limit.
Thus, by Zorn's lemma, $[f]$ has a minimal element $h$.
By definition $h \le f$ is convex and, as $f_{0} \in [f]$, for all $\lVert y \rVert = \lVert x_{0} \rVert$, $h((1-\lambda)y) \le h((1-\lambda)x_{0})$.

Let $\lambda = \delta / \lVert x_{0} \rVert$.
For any direction $d$ there exists an orthogonal matrix $U$ such that
\begin{equation*}
d = - \delta U \frac{x_{0}}{\lVert x_{0} \rVert}
= - \lambda U x_{0}.
\end{equation*}
Then
\begin{equation*}
(1-\lambda) x_{0} = (1-\lambda) x_{0} +  \lambda (I - U)x_{0} + \lambda( U - I) x_{0}
\end{equation*}
and
\begin{align*}
h((1-\lambda)x_{0})
&\le (1 - \lambda ) h( (1- \lambda) x_{0} + \lambda (I - U)x_{0}) \\
&\quad+ \lambda h((1-\lambda) x_{0} + \lambda (I - U)x_{0} + (U - I)x_{0}) \\
&= (1- \lambda) h( x_{0} - \lambda U x_{0}) + \lambda h((1-\lambda) U x_{0}) \\
&\le (1- \lambda) h( x_{0} - \lambda U x_{0}) + \lambda h((1-\lambda) x_{0}),
\end{align*}
the first inequality uses the convex combination $(1- \lambda )0 + \lambda (U - I) x_{0}$ and the second is the definition of $h$.
Rearranging, by definition of $h$,
\begin{equation*}
f((1-\lambda)x_{0})
= h((1-\lambda)x_{0})
\le h(x_{0} - \lambda U x_{0})
\le f(x_{0} - \lambda U x_{0}).
\end{equation*}
Hence
\begin{equation*}
f(x_{0} + \delta d_{0})
= f((1- \lambda) x_{0})
\le f( x_{0} - \lambda U x_{0})
= f(x_{0} + d).
\end{equation*}

For the second statement, suppose there are two minimizers, $d_{1} \ne d_{2}$ of norm $\delta$, with
\begin{equation*}
f(x_{0} + d_{1}) = f(x_{0} + d_{2}),
\end{equation*}
then, for all $0 \le \beta \le 1$, by convexity,
\begin{equation*}
x_{0} + \beta d_{1} + (1- \beta) d_{2}
\end{equation*}
is also a minimizer.
However, if $0 < \beta < 1$, this is an interior point of the ball of radius $\delta$ around $x_{0}$, a contradiction.
\end{proof}

Along this optimal direction $f$ is strictly decreasing.

\begin{lemma}
\label{lemma_strictly_decreasing}
Let $f$ be convex with a unique minimizer $x^{\ast}$, $x_{0} \ne x^{\ast}$ in $\mathrm{dom}(f)$ and
\begin{equation*}
z_{\alpha} = x_{0} + \alpha ( x^{\ast} - x_{0}).
\end{equation*}
For all $0 \le \alpha < \beta \le 1$,
\begin{equation*}
f(z_{\beta}) < f(z_{\alpha}).
\end{equation*}
\end{lemma}

\begin{proof}
\begin{equation*}
z_{\beta} = z_{\alpha} + (\beta - \alpha)(x^{\ast} - x_{0})
\end{equation*}
and
\begin{equation*}
x^{\ast} - z_{\alpha} = (1- \alpha)(x^{\ast} - x_{0}).
\end{equation*}
Thus, letting $\gamma = \frac{ \beta - \alpha }{ 1 - \alpha}$ with $0 < \gamma \le 1$,
\begin{equation*}
z_{\beta} = \gamma x^{\ast} + (1- \gamma) z_{\alpha}.
\end{equation*}
Hence
\begin{equation*}
f(z_{\beta})
\le \gamma f(x^{\ast}) + (1-\gamma) f(z_{\alpha})
< f(z_{\alpha}).
\end{equation*}
\end{proof}

\section{Optimization}

We consider the following optimization problem in $\mathbb{R}^{n}$:
\begin{equation*}
\inf f(x),
\end{equation*}
where $f : \mathbb{R}^{n} \rightarrow (-\infty,\infty]$ is $K$-Lipschitz over $\mathrm{dom}(f)$ with a unique minimizer and $\mathrm{dom}(f)$ is compact.

Combing the results of the previous sections, the lower convex envelope $\breve{f}$ has the same unique minimizer and any $x \in A_{f} \cap \mathrm{int}(\mathrm{dom}(f))$ has a subgradient $g_{x}$.

\subsection{Algorithm}

This section gives a formal statement of our algorithm.

\begin{algorithm}
{\rm (Directional descent)}

Stage 1:
\begin{equation*}
d^{\ast} = \arg\min_{ \lVert d \rVert \le \delta } f(x_{0} + d),  
\end{equation*}
where $x_{0} \in \mathrm{dom}(f)$ and either $x_{0} + \delta d_{0} \in A_{f}$ or $\delta \ge \lVert x^{\ast} - x_{0} \rVert$.

Stage 2:
\begin{equation*}
x_{m} = x_{0} + m \alpha \frac{ d^{\ast}}{ \delta },
\end{equation*}
for some step size $\alpha > 0$.
\end{algorithm}

If the minimizer happens to be in the $\delta$-ball of stage 1, i.e. $\delta \ge \lVert x^{\ast} - x_{0} \rVert$, then stage 2 is unnecessary.
If this is not the case, then, by Theorem \ref{theorem_optimal_direction}, $d^{\ast} = \delta d_{0}$.
While $x_{0} + \delta d_{0} \in A_{f}$ is all that is required for this property to hold, it may be more convenient for the entire ball of radius $\delta$ around $x_{0}$ to be contained in $A_{f}$, providing a convex optimization.
The convergence guarantees of stage 2 follow directly from $f$ being $K$-Lipschitz.
For any direction $d^{\ast} = \delta d$, such that $x_{m}$ remains in the effective domain,
\begin{align*}
f(x_{m}) - f(x^{\ast})
&\le K \lVert x_{m} - x^{\ast} \rVert \\
&= K \lVert m \alpha d - (x^{\ast} - x_{0}) \rVert \\
&\le K \lVert m \alpha d_{0} - (x^{\ast} - x_{0}) \rVert  + Km \alpha \lVert d - d_{0} \rVert\\
&= K \left \lvert m \alpha - \lVert x^{\ast} - x_{0} \rVert \right \rvert + Km\alpha \lVert d - d_{0} \rVert \\
&\le K \alpha + K\lVert x^{\ast} - x_{0} \rVert  \lVert d - d_{0} \rVert
\quad \left( m = \left \lfloor \frac{\lVert x^{\ast} - x_{0} \rVert}{\alpha} \right \rfloor  \right).
\end{align*}

This gives an $O( 1/\varepsilon )$ algorithm when $d$ is appropriately close to $d_{0}$.
We note that, without some local strong convexity type of structural property, the function may be arbitrarily flat.

\section{Background}

\begin{theorem}
\label{theorem_caratheodory}
{\rm (Caratheodory)}
Let $X$ be a nonempty subset of $\mathbb{R}^{n}$.
Then every vector from $\mathrm{conv}(X)$ is a convex combination of at most $n+1$ vectors from $X$.
\end{theorem}

\begin{proof}
\cite{Bertsekas2009} Proposition 1.2.1.
\end{proof}

\begin{proposition}
\label{proposition_subgradient_existence}
Let $f : \mathbb{R}^{n} \rightarrow (-\infty,\infty]$ be a proper convex function.
If $x \in \mathrm{int}(\mathrm{dom}(f))$, then $\partial f(x)$ is nonempty and compact.
\end{proposition}

\begin{proof}
\cite{Bertsekas2009} Proposition 5.4.1.
\end{proof}

\section{Deferred proofs}

\subsection{Proof of Lemma \ref{lemma_lower_convex_envelope}}

\begin{proof}
Let $C(x) = C(x,f)$.
Then $x \in \mathrm{conv}(\mathrm{dom}(f))$ if and only if $C(x)$ is nonempty, and $g(x)=\infty$ when $C(x)$ is empty.
Moreover, if $x \in \mathrm{dom}(f)$, then $(1,x) \in C(x)$ and $g(x) \le f(x)$.
If $h \le f$ is convex, then, for all $(\underline{\lambda},\underline{x}) \in C(x)$,
\begin{equation*}
h(x) \le \sum_{k=1}^{m} \lambda_{k} h(x_{k})
\le \sum_{k=1}^{m} \lambda_{k} f(x_{k}).
\end{equation*}
Thus $h \le g$ and it suffices to show that $g$ is convex.
To that end, let $z = \alpha x + (1-\alpha) y$ for some $x,y \in \mathrm{conv}(\mathrm{dom}(f))$ and $0 \le \alpha \le 1$.
Then
\begin{equation*}
C(z) \supset \{ (\alpha \underline{\lambda}^{(1)},\underline{x}) : (\underline{\lambda}^{(1)},\underline{x}) \in C(x) \} \oplus \{  ( (1-\alpha) \underline{\lambda}^{(2)},\underline{y}) : (\underline{\lambda}^{(2)},\underline{y}) \in C(y) \}.
\end{equation*}
Thus $g(z)$ is upper bounded by
\begin{equation*}
\inf \left \{ \sum_{i=1}^{ m_{1}} \alpha \lambda_{i}^{(1)} f(x_{i}) \! + \! \sum_{j=1}^{m_{2}} (1 \! - \! \alpha) \lambda_{j}^{(2)} f(y_{j}) : (\underline{\lambda}^{(1)},\underline{x}) \in C(x), (\underline{\lambda}^{(2)},\underline{y}) \in C(y)  \right\}.
\end{equation*}
Hence
\begin{equation*}
g(\alpha x + (1- \alpha) y) \le \alpha g(x) + (1- \alpha)g(y).
\end{equation*}
\end{proof}

\end{document}